\documentclass[english]{amsart}
\usepackage{babel}
\usepackage{amsmath}
\usepackage{setspace}
\usepackage{amssymb}
\usepackage[colorlinks,citecolor=red,pagebackref,hypertexnames=false]{hyperref}
\usepackage[backrefs]{amsrefs}

\makeatletter

\newtheorem{theorem}{Theorem}[section]
\newtheorem{definition}{Definition}[section]
\newtheorem{observ}{Observation}[section]

\newtheorem{lemma}[theorem]{Lemma}
\newtheorem{corollary}[theorem]{Corollary}

\newtheorem{discussion}[observ]{Discussion}
\newtheorem{claim}[theorem]{Claim}

\newtheorem{convention}{Convention}

\newtheorem{fact}{Fact}

\newenvironment{proof1}{proof}{\hspace{2cm}$\square$\par}

\begin{document}

\title{Stable theories and representation over sets}

\author[M. Cohen]{Moran Cohen}
\address{
Institute of Mathematics The Hebrew University of Jerusalem, Jerusalem
91904, Israel.
}
\email{moranski@math.huji.ac.il}

\author[S. Shelah]{Saharon Shelah }
\address{
Institute of Mathematics The Hebrew University of Jerusalem, Jerusalem
91904, Israel and Department of Mathematics Rutgers University New Brunswick,
NJ 08854, USA
}
\email{shelah@math.huji.ac.il}

\thanks{Publication 919 of the second author}

\begin{abstract}
In this paper we explore the representation property over sets. This property generalizes constructibility, however is weak enough to enable us to prove that the class of theories $T$ whose models are representable is exactly the class of stable theories.
Stronger results are given for $\omega$-stable.

\end{abstract}

\maketitle

\section{Preliminaries}

\begin{convention}
We use $\mathfrak{k}$ to denote an arbitrary class of structures
(of a given language, closed under isomorphism). The class of structures
of the language $\left\{ =\right\} $ is denoted $\mathfrak{k}^{{\rm eq}}$.
\begin{enumerate}
\item $\mathfrak{C}$ is a {}``monster'' model for $T$. i.e. a sufficiently
saturated one.
\item for a sequence of sets $\left\langle A_{\beta}:\beta<\alpha\right\rangle $
let $A_{<\alpha}:=\bigcup_{\beta<\alpha}A_{\beta}$ , $A_{\leq\alpha}:=A_{<\alpha+1}$.
\item  $\mathrm{tp}(a,A):=\mathrm{tp}(a,A,\mathfrak{C})$.

\end{enumerate}
\end{convention}

\definition{Let $\mathfrak{k}$ be a class of structures of a given vocabulary
$\tau$.}

\begin{itemize}
\item ${\rm Ex}_{\mu,\kappa}^{1}(\mathfrak{k})$ denotes the minimal class
of structures $\mathfrak{k}^{\prime}\supseteq\mathfrak{k}$ with the
property that for each structure $I\in\mathfrak{k}$ there exists an
enrichment $I^{+}\in\mathfrak{k}^{\prime}$ by a partition $\left\langle P_{\alpha}^{I^{+}}:\alpha<\kappa\right\rangle $,
partial unary functions $\left\langle F_{\beta}^{I^{+}}:\beta<\mu\right\rangle $
such that $F_{\beta}(P_{\alpha})\subseteq P_{<\alpha}$ and $P_{\alpha},F_{\beta}\notin\tau_{I}$
hold for every $\alpha<\kappa,\ \beta<\mu$

\item For a given model $I\in\mathfrak{k}$, 
we define the free algebra $\mathcal{M}=\mathcal{M}_{\mu,\kappa}(I)$ as the model
having the language $\tau^+:=\tau_{I}\cup\left\{ F_{\alpha,\beta},\right\}_{\alpha<\mu,\;\beta<\kappa}$
where each $F_{\alpha,\beta}$ is a $\beta$-place function.
$\left\|\mathcal{M}_{\mu,\kappa}(I)\right\|$ consists of all the terms constructed in the usual 
(well-founded inductive) way from elements of $I$ using the functions $F_{\alpha,\beta}$
The functions and relations of $\tau_I$ are interpreted in $\mathcal{M}$ as partial functions and restricted relations on $I\subseteq\mathcal{M}$.

\item Let $\theta_{\mu,\kappa}:=\left|\mathcal{M}_{\mu,\kappa}(\kappa)\right|$
(The power of the set of $\mu$-terms with $\kappa$ constants).

\item ${\rm Ex}_{\mu,\kappa}^{2}(\mathfrak{k})$ denotes the class of models 
of the form $\mathcal{M}_{\mu,\kappa}(I)$ for every $I\in\mathfrak{k}$. 

\end{itemize}

\definition{\label{def:weak_repr} Let $M\models T$, $I$ a structure. $f:M\rightarrow I$
is a ($\Gamma,\Delta$)-representation of $M$ in-$I$ iff ${\rm Rang}(f)$
is closed under functions in $I$ (both partial and full), and for
$\overline{a},\overline{b}\in\ ^{<\omega}M$ the following holds:
\[
{\rm tp}_{\Gamma}(f(\overline{a}),\emptyset,I)={\rm tp}_{\Gamma}(f(\overline{b}),\emptyset,I)\;\Rightarrow\;{\rm tp}_{\Delta}(\overline{a},\emptyset,M)={\rm tp}_{\Delta}(\overline{b},\emptyset,M)\]
}

\begin{itemize}
\item We say that $M$ is $(\mathfrak{k},\Gamma,\Delta)$-representable
if $I\in\mathfrak{k}$ and there exists a $(\Gamma,\Delta)$-representation
$f:M\to I$ .
\item we say that the theory $T$ is $(\mathfrak{k},\Gamma,\Delta)$-representable
if every $M\models T$ is $(\mathfrak{k},\Gamma,\Delta)$-representable.
\item We omit $\Delta,\Gamma$ from the notation if $\Gamma=qf_{\mathcal{L}[\tau_{I}]},\;\Delta=\mathcal{L}[\tau_{M}]$.
\end{itemize}

\observ{let $M\models T$}

\begin{itemize}
\item $M$ is $\mathfrak{k}$-representable implies $M$ is ${\rm Ex}_{\mu,\kappa}^{i}(\mathfrak{k})$-representable $(i=1,2)$.
\item For $i=1,2$: $M$ is ${\rm Ex}_{\mu_{2},\kappa_{2}}^{i}({\rm Ex}_{\mu_{1},\kappa_{1}}^{i}(\mathfrak{k}))$-representable
iff $M$ is ${\rm Ex}_{\mu_{1}+\mu_{2},\kappa_{1}+\kappa_{2}}^{i}(\mathfrak{k})$-representable.
\item $M$ is ${\rm Ex}_{\mu_{2},\kappa_{2}}^{2}({\rm Ex}_{\mu_{1},\kappa_{1}}^{1}(\mathfrak{k}))$-representable
implies $M$ is ${\rm Ex}_{\mu_{1},\kappa_{1}}^{1}({\rm Ex}_{\mu_{2},\kappa_{2}}^{2}(\mathfrak{k}))$-representable.
\end{itemize}

\fact{\label{fac:m_k_eq_repr}A map $f:M\to I^{+}$ is a representation
of $M$ in $I^{+}\in{\rm Ex}_{\mu,\kappa}^{1}(\mathfrak{k}^{{\rm eq}})$,
if $\mathrm{tp}(\bar{a},\emptyset,M)=\mathrm{tp}(\bar{b},\emptyset,M)$
holds for every $U,\tilde{h},\overline{a},\overline{b}$ fulfilling
the following condition:}

$U\subseteq\left|I^{+}\right|$ is such that $\mathrm{cl}_{\left\{ F_{\beta}^{I^{+}}\right\} }U=U$,
$\tilde{h}$ is a partial automorphism of $I^{+}$ whose domain contains
$U$, and $\bar{a},\bar{b}\in\ ^{m}M$ are sequences such that $\tilde{h}(f(\bar{a}))=f(\bar{b})$
and $f(\overline{a})\subseteq U$.

\section{Stable Theories}

\discussion{In this section we prove the equivalence stable = ${\rm Ex}_{\mu_{1},\kappa_{1}}^{1}({\rm Ex}_{\mu_{2},\kappa_{2}}^{2}(\mathfrak{k}^{{\rm eq}}))$-representable}

\theorem{\label{thm:find_indep_set}Let $T$ be ${\rm Ex}_{\mu_{1},\kappa_{1}}^{1}({\rm Ex}_{\mu_{2},\kappa_{2}}^{2}(\mathfrak{k}^{{\rm eq}}))$-representable.
If $\overline{b}=\left\langle \overline{b}_{\alpha}:\alpha<\lambda\right\rangle \subseteq\mathfrak{C}$,
is such that $\lg\overline{b}_{\alpha}<\mu=\mu_{1}+\kappa_{2},\;\lambda>\kappa_{1}+\mu_{1}+\kappa_{2}+$,
and $\lambda>\chi^{<\mu}$ for every $\chi<\lambda$ then there exists
an $S\subseteq\lambda$ of cardinality $\lambda$ such that $\left\langle \overline{b}_{\alpha}:\alpha\in S\right\rangle $
is an indiscernible set.}

\begin{proof}
Let $M\models T$ contain $\overline{{\bf a}}$, $f:M\to\mathcal{M}^{++}:=\left(\mathcal{M}_{\mu_{2},\kappa_{2}}(I),P_{\alpha},F_{\beta}\right)_{\alpha<\kappa_{1},\beta<\mu_{1}}$
a representation. Let $\overline{a}_{\alpha}=f(\overline{b}_{\alpha})$.

assume w.l.o.g:
\begin{itemize}
\item Every $\overline{a}_{\alpha}$ is closed under subterms in $\mathcal{M}_{\mu_{2},\kappa_{2}}(I)$.
\item Every $\overline{a}_{\alpha}$ is closed under the partial functions
$F_{\beta}$.
\item $\lg\overline{a}_{\alpha}=\xi$ ( for all $\alpha<\lambda$ ).
\end{itemize}
$\lambda={\rm cf}\lambda>\left(\theta_{\mu_{1},\kappa_{1}}\right)^{\xi}$
and therefore there exists $\overline{\sigma}(\overline{x}),\;\lg\overline{x}<\kappa_{2},\; S_{0}\in\left[\lambda\right]^{\lambda}$
such that for all $\alpha\in S_{0}$ there exists $\overline{t}_{\alpha}\subseteq\ ^{<\kappa_{2}}I$
such that $\overline{a}_{\alpha}=\overline{\sigma}(\overline{t}_{\alpha})$.

similarly, $\left(\kappa_{2}\right)^{\xi}<\lambda$ and there exists
an $S_{1}\in\left[S_{0}\right]^{\lambda}$ on which the map \[
\alpha\mapsto\left\{ \left(i,\beta\right)\in\xi\times\kappa_{2}:a_{\alpha}^{i}\in P_{\beta}\right\} \]
is constant and equal to a binary relation $R_{1}$.

Also, $\xi^{\mu_{1}+\xi}<\lambda$ implies that there exists an $S_{2}\in\left[S_{1}\right]^{\lambda}$
on which the map

\[
\alpha\mapsto\left\{ \left(\beta,\zeta_{0},\zeta_{1}\right):\zeta_{0},\zeta_{1}<\xi,\;\beta<\mu_{1},\; F_{\beta}(a_{\alpha}^{\zeta_{0}})=a_{\alpha}^{\zeta_{1}}\right\} \]

is constant and equal to the relation $R_{2}$.

\relax From lemma \ref{thm:delta_sys_lemma} it follows that there exist
$S_{3}\in\left[S_{2}\right]^{\lambda}$, $U\subseteq\xi,\; E\subseteq\xi\times\xi$
such that:
\begin{itemize}
\item $\overline{a}_{\alpha}\upharpoonright U=\overline{a}_{\beta}\upharpoonright U$
for all $\alpha,\beta\in S_{3}$
\item $E$ an equvalence such that $a_{\alpha}^{i}=a_{\alpha}^{j}\leftrightarrow(i,j)\in E$
for all $\alpha\in S_{3}$.
\item $a_{\alpha}^{i}=a_{\beta}^{j}\to i,j\in U$ for all $\alpha\neq\beta\in S_{3}$
.
\end{itemize}
We show that for every $\overline{u},\overline{v}\subseteq S_{3}$
without repetitions and of length $\ell<\omega$ , there exists a
partial automorphism $h$ of $\mathcal{M}^{++}$ such that $h(\overline{a}_{\overline{v}})=\overline{a}_{\overline{u}}$.

Indeed, define $h(a_{v_{k}}^{j})=a_{u_{k}}^{j}$for all $j<\xi,k<\ell$.
$E$ and $U$ show that $a_{v_{k_{0}}}^{j_{0}}=a_{v_{k_{1}}}^{j_{1}}\to a_{u_{k_{0}}}^{j_{0}}=a_{u_{k_{1}}}^{j_{1}}$.
Hence, $h$ is well-defined.

Let the term $\overline{\sigma}(\overline{t})$ be in ${\rm Dom}(h)$.
Since $\overline{a}_{\overline{v}}$ is closed under subterms it follows
that $h(\overline{\sigma}(\overline{t}))=\overline{\sigma}(h(\overline{t}))$.

$h$ respects $P_{\alpha}$:

\[
a_{u_{\kappa}}^{j}\in P_{\alpha}\leftrightarrow(j,\alpha)\in R_{1}\leftrightarrow a_{v_{k}}^{j}\in P_{\alpha}\]

$h$ commutes with $F_{\alpha}$: for all $a_{v_{k_{0}}}^{j_{0}},a_{v_{k_{1}}}^{j_{1}}\in{\rm Dom}(h)$,
since $\overline{a}_{v_{k_{0}}}$ is closed under it follows that
there exists $j<\xi$ such that $F_{\alpha}(a_{v_{k_{0}}}^{j_{0}})=a_{v_{k_{0}}}^{j}$.
Therefore $(j,j_{0})\in E$ and by the definition it follows \[
F_{\alpha}(h(a_{v_{k_{0}}}^{j_{0}}))=F_{\alpha}(a_{v_{k_{1}}}^{j_{0}})=a_{v_{k_{1}}}^{j}=h(a_{v_{k_{0}}}^{j})=h(F_{\alpha}(a_{v_{k_{0}}}^{j_{0}}))\]

\end{proof}

\theorem{\label{thm:x.1}Quote theorem \cite{Sh:c}*{II.2.13}}

\theorem{\label{thm:eq_repr_is_stable} $T$ is ${\rm Ex}_{\mu_{1},\kappa_{1}}^{1}({\rm Ex}_{\mu_{2},\kappa_{2}}^{2}(\mathfrak{k}^{{\rm eq}}))$-representable
implies $T$ stable.}

\begin{proof}
Let $T$ be unstable. from theorem \ref{thm:x.1} and compactness,
there exist $\varphi(\overline{x},\overline{y})\in\mathcal{L}_{T}$,
$M\models T$ and a sequence $\left\langle \overline{a}_{i}:i<\lambda\right\rangle $,
$\lambda=\left(\kappa^{\mu}\right)^{+}+\beth_{2}(\mu)^{+}$ such that
$\models\varphi(\overline{a}_{i},\overline{a}_{j})^{\mathrm{if}(i<j)}$
for all $i,j<\lambda$. Assume towards contradiction that $f:M\to I^{+}$
is a representation of $M$ in $I^{+}\in{\rm Ex}_{\mu_{1},\kappa_{1}}^{1}({\rm Ex}_{\mu_{2},\kappa_{2}}^{2}(\mathfrak{k}^{{\rm eq}}))$.
Then theorem \ref{thm:find_indep_set} implies in particular the
existence of $i,j<\lambda$, a partial automorphism $g$ of $I^{+}$
with domain and range closed under functions, such that: 
$$
g(f(\overline{a}_{i}\!^\frown\overline{a}_{j}))=f(\overline{a}_{j}\!^\frown\overline{a}_{i})
$$

from the definition of a representation we get 
$$
{\rm tp}(\overline{a}_{i}\!^\frown\overline{a}_{j},\emptyset,M)=
{\rm tp}(\overline{a}_{j}\!^\frown\overline{a}_{i},\emptyset,M)
$$
a contradiction to the definition of $\varphi$.
\end{proof}

\discussion{We now turn to the proof of the other direction of equivalence. This
will require more facts on stable theories and strongly independent
sets, defined below.}

\definition{A set $\mathbb{I}\subseteq\mathfrak{C}$ will be called strongly
independent over $A$ if the following holds:}

$\circledast$ for all $a\in\mathbb{I}$, $\mathrm{tp}(a,A\cup\mathbb{I}\backslash\left\{ a\right\} ,M)$
is the unique $p\in\mathbf{S}(A\cup\mathbb{I}\backslash\left\{ a\right\} )$
such that $p\supseteq\mathrm{tp}(a,A\cup\mathbb{I}\backslash\left\{ a\right\} )$
and $p$ does not fork over $A$.

\definition{\label{def:decompo}We call the a sequence $\left\langle \mathbb{I}_{\alpha}:\alpha<\mu\right\rangle $
of subsets of $M$ a strongly-independent decomposition (in short:
s.i.d) of length $\mu$ of $M$ if for all $\alpha<\mu$, it holds
that $\mathbb{I}_{\alpha}$ is strongly independent over $\mathbb{I}_{<\alpha}$,
and that $\left|M\right|=\mathbb{I}_{<\mu}$.\cite{Sh:c}*{II.2.13}}

\begin{convention}
We assume from this point onwards that $T$ is stable.
\end{convention}
\begin{claim}
\label{cla:str_indep_symmetry}let $a_{1},a_{2}\in\mathfrak{C}$,
$A\supseteq B_{1},B_{2}$ such that $\mathrm{tp}(a_{i},A\cup\{ a_{3-i}\})$
is non-forking over $B_{i}$, and $\mathrm{tp}(a_{i},A)$ is the unique
nonforking extension in $\mathbf{S}(A)$ of $\mathrm{tp}(a_{i},B_{i})$.
Then $(\ast)_{1}\Leftrightarrow(\ast)_{2}$ where:
\end{claim}
$(\ast)_{i}$ $\mathrm{tp}(a_{i},B_{i})$ has a unique nonforking
extension whose domain is $A\cup\{ a_{3-i}\}$.

\begin{proof}
it is sufficient to prove $\neg(\ast)_{2}\Rightarrow\neg(\ast)_{1}$,
since the converse follows by symmetry. 

Assume that $\mathrm{tp}(a_{2},B_{2})$ has two distinct nonforking
extensions $p_{1},p_{2}\in\mathbf{S}(A\cup\{ a_{1}\})$. 

Then, there exists $\varphi\in p_{1},\ \neg\varphi\in p_{2},\ \varphi=\varphi(x,a_{1},\bar{c})$.
Let $b_{1},b_{2}$ realize $p_{1},p_{2}$, respectively. 

$\mathrm{tp}(b_{i},A)=p_{i}\upharpoonright A$ is a nonforking extension
of $p$ implies $p_{1}\upharpoonright A=p_{2}\upharpoonright A$.
Thus, for $i<2$ There exist elementary maps $F_{i}$ in $\mathfrak{C}$
so that $F_{i}\upharpoonright A=\mathrm{id}_{A},\; F_{i}(b_{i})=a_{2}$.

Let $q_{i}\in\mathbf{S}(A\cup\{ b_{i}\})$ be a nonforking extension
of $\mathrm{tp}(a_{1},B_{1})$. 

Then $F_{i}(q_{i})\in\mathbf{S}(A\cup\{ a_{2}\})$ is a nonforking
extension of $\mathrm{tp}(a_{1},B_{1})$ ($F_{i}\upharpoonright A=\mathrm{id}_{A}$,
and elementary maps preserve nonforking).

Now note that $\models\varphi(b_{1},a_{1},\bar{c})\wedge\neg\varphi(b_{2},a_{1},\bar{c})$,
hence $\varphi(a_{2},x,\bar{c})\in F_{1}(q_{1})$ and $\neg\varphi(a_{2},x,\bar{c})\in F_{2}(q_{2})$.
Therefore, $F_{i}(q_{i})$ are distinct extensions, as required.
\end{proof}

\definition{An ordered partition $\left\langle \mathbb{J}_{\alpha}:\alpha<\mu^{\prime}\right\rangle $
is called an order-preserving refinement of the ordered partition
$\left\langle \mathbb{I}_{\alpha}:\alpha<\mu^{\prime}\right\rangle $
if it is a refinement as a partition and $\alpha^{\prime}<\beta^{\prime}$
for all $\alpha<\beta<\mu,\;\alpha^{\prime},\beta^{\prime}<\mu^{\prime}$
such that $\mathbb{I}_{\alpha}\supseteq\mathbb{J}_{\alpha^{\prime}}\;\mathbb{I}_{\beta}\supseteq\mathbb{J}_{\beta^{\prime}}$.}

\begin{claim}
\label{cla:sub_decompo}If $\left\langle \mathbb{I}_{\alpha}:\alpha<\mu\right\rangle $
is an s.i.d of $M$, then every order-preserving refinement of it
is an s.i.d. of $M$. 
\end{claim}
\begin{proof}
Let $\alpha^{\prime}<\mu^{\prime}$. we show that $\mathbb{J}_{\alpha^{\prime}}\subseteq\mathbb{I}_{\alpha}$
is strongly independent over $\mathbb{J}_{<\alpha^{\prime}}$. Then
let $a\in\mathbb{J}_{\alpha^{\prime}}$. 

$\mathrm{tp}(a,\mathbb{I}_{\leq\alpha}\backslash\left\{ a\right\} )$
is nonforking over $\mathbb{I}_{<\alpha}$ and hence, the reduct $\mathrm{tp}(a,\mathbb{J}_{\leq\alpha^{\prime}})$
is nonforking over $\mathbb{I}_{<\alpha}$, nor does it fork over
the larger $\mathbb{J}_{<\alpha^{\prime}}$. On the other hand, if
$\mathrm{tp}(a,\mathbb{J}_{<\alpha^{\prime}})\subseteq q\in\mathbf{S}(\mathbb{J}_{\leq\alpha^{\prime}}\backslash\left\{ a\right\} )$
is nonforking over $\mathbb{J}_{<\alpha}$, it has an extension $q\subseteq q^{\prime}\in\mathbf{S}(\mathbb{I}_{\leq\alpha}\backslash\left\{ a\right\} )$
which is nonforking over $\mathbb{J}_{<\alpha}$. $a\in\mathbb{I}_{\alpha}$
implies that $\mathrm{tp}(a,\mathbb{J}_{\leq\alpha^{\prime}}\backslash\left\{ a\right\} )\subseteq\mathrm{tp}(a,\mathbb{I}_{\leq\alpha}\backslash\left\{ a\right\} )$
is nonforking over $\mathbb{I}_{<\alpha}$ and since $\mathbb{I}_{\alpha}$
is strongly independent over $\mathbb{I}_{<\alpha}$ we get $q^{\prime}=\mathrm{tp}(a,\mathbb{I}_{\leq\alpha}\backslash\left\{ a\right\} )$,
and in particular, \[
q^{\prime}\upharpoonright\left(\mathbb{J}_{\leq\alpha^{\prime}}\backslash\left\{ a\right\} \right)=\mathrm{tp}(a,\mathbb{J}_{\leq\alpha^{\prime}}\backslash\left\{ a\right\} )\]

therefore $\mathbb{J}_{\alpha^{\prime}}$ is as required, nonforking
over $\mathbb{J}_{<\alpha^{\prime}}$.
\end{proof}

\theorem{\label{thm:nonforking_contradicting_types} Let $p,q\in\mathbf{S}(B)$
be distinct, nonforking over $A\subseteq B$. Then there exists an
$E\in\mathrm{FE}(A)$ such that: \[
p(x)\cup q(y)\vdash\neg E(x,y)\]
}

(cf\.{ }\cite{Sh:c}*{III;2.9(2)}

\begin{claim}
\label{cla:nonforking_contr_contr}Let $A\subset B$ be such that
if $\varphi$ is a formula over $B$ which is almost over $A$, then
there exists a formula over $A$ which is equivalent to $\varphi$
modulo $T$. If $p,q\in\mathbf{S}(B)$ are distinct nonforking over
$A$, There exists a $\varphi_{\ast}(x,\overline{c})$ such that $p\vdash\varphi_{\ast},\; q\vdash\neg\varphi_{\ast}$.
\end{claim}
\begin{proof}
By \ref{thm:nonforking_contradicting_types}, there exists an $E\in FE(A)$
such that $p(x)\cup q(y)\vdash\neg E(x,y)$. Let $\left\{ b_{i}:i<n(E)\right\} \subseteq\mathfrak{C}$
represent the equivalence classes of $E$. Define $w:=\left\{ i<n(E):p(x)\cup\left\{ E(x,b_{i})\right\} \ \textrm{is consistent}\right\} $,
and let $\varphi(x):=\bigvee_{i\in w}E(x,b_{i})$. Then
\begin{itemize}
\item w.l.o.g for all $i\in w$, $b_{i}\in\mathfrak{C}$ realizes $p$.
\item $p(x)\vdash\varphi(x)$ ( if $a$ realizes $p$ there exists a $b_{i}$
such that $\models aEb_{i}$ since the $b_{i}$ are representatives
of the equivalence classes of $E$. on the other hand, $i$ must belong
to $w$, which implies that $\varphi(a)$ holds ) .
\item Similarly, $q(x)\vdash\neg\varphi(x)$ since if $a$ realizes $q$
then $p(x)\cup q(y)\vdash\neg E(x,y)$ therefore $\neg E(b_{i},a)$
for all $i\in w$, therefore $\models\neg\varphi(a)$.
\item $\varphi(x)$ is preserved by members of ${\rm Aut}(\mathfrak{C},B)$:
Let $f\in{\rm Aut}(\mathfrak{C},B)$. Then $f$ preserves $E$ (and
its equivalence classes in $\mathfrak{C}$) and $p$ ($\mathrm{Dom}(p)=B$)
implying:

\begin{itemize}
\item $p(x)\cup\left\{ E(x,b_{i})\right\} \Leftrightarrow p(x)\cup\left\{ E(x,f(b_{i}))\right\} $
holds for all $i<n(E)$.
\item $\neg E(f(b_{i}),f(b_{j}))$ for all $i,j<n(E),\; i\neq j$.
\item $f$ acts as a permutation on $\mathfrak{C}/E$, and when reduced
also on $\left\{ b_{i}/E:i\in w\right\} $, therefore: \[
f(\varphi(\mathfrak{C}))=f(\bigcup_{i\in w}b_{i}/E)=\bigcup_{i\in w}f(b_{i})/E=\varphi(\mathfrak{C})\]

\end{itemize}
\end{itemize}
implying $\models\varphi(x)\equiv f(\varphi(x))$. Lemma Sh:c,III.2.3]{[}
implies that $\varphi(x)$ has an equivalent formula $\varphi_{\ast}$
over $B$, as needed.
\end{proof}
\begin{claim}
\label{cla:kappa_T}For all $p\in\mathbf{S}^{m}(B)$ there exists
$A\subseteq B,\;\left|A\right|<\kappa(T)$ such that $p$ does not
fork over $A$. Also, $\kappa(T)\leq\left|T\right|^{+}$.
\end{claim}
(cf. \cite{Sh:c}*{III;3.2, 3.3})

\begin{claim}
\label{cla:count_almost_over}The number of formulas almost over $A$
is (up to logical equivalence) at most $\left|A\right|+\left|T\right|$
\end{claim}
(cf. \cite{Sh:c}*{III;2.2(2)})

\lemma{\label{lem:stable_exist_decomp} if $M\models T$ then there exists
an s.i.d of length $\mu=\left|T\right|^{+}$}

\begin{proof}
We construct inductively a sequence $\left\langle \mathbb{I}_{\alpha}:\alpha<\mu\right\rangle $
such that $\mathbb{I}_{\alpha}$ is strongly independent over $\mathbb{I}_{<\alpha}$
and is moreover maximal with respect to this property ( for all $\mathbb{I}\supseteq\mathbb{I}_{\alpha}$
is not strongly independent over $\mathbb{I}_{<\alpha}$), for all
$\alpha$.

Assume towards contradiction that $a\in M\backslash\mathbb{I}_{<\mu}$.
By the definition of $\kappa(T)$ and \ref{cla:kappa_T} we get a
set \[
B\subseteq\mathbb{I}_{<\mu},\;\left|B\right|<\kappa(T)\leq\left|T\right|^{+}\]
such that $p(x):=\mathrm{tp}(a,\mathbb{I}_{<\mu})$ is nonforking
over $B$, and there exists an $\alpha_{0}(\ast)<\mu$ such that $\mathbb{I}_{<\alpha_{0}(\ast)}\supseteq B$.

Let \[
\Gamma:=\left\{ \varphi(x;\bar{c}):\varphi(x,\bar{c})\textrm{ is almost over }B,\;\varphi(x;\bar{y})\in\mathcal{L},\;\bar{c}\in^{\lg\bar{y}}\mathbb{I}_{<\mu}\right\} \]

By claim \ref{cla:count_almost_over}, there exists a set $\Gamma_{\ast}\subseteq\Gamma,\;\left|\Gamma_{\ast}\right|\leq\left|B\right|+\left|T\right|<\mathrm{cf}(\left|T\right|^{+})$
of representatives (by logical equivalence) of the formulas almost
over $B$. Hence, there exists $\alpha_{1}(\ast)<\mu$ such that $\bar{b}\subseteq\mathbb{I}_{<\alpha(\ast)}$
for all $\varphi(x,\bar{b})\in\Gamma_{\ast}$ . Let $\alpha(\ast)=\max_{i<2}\left\{ \alpha_{i}(\ast)\right\} $. 

We now show that $p\upharpoonright\mathbb{I}_{\leq\alpha(\ast)}$
is the only extension in $\mathbf{S}(\mathbb{I}_{\leq\alpha(\ast)})$
of $p\upharpoonright\mathbb{I}_{<\alpha(\ast)}$ which is nonforking
over $\mathbb{I}_{<\alpha(\ast)}$:

Indeed, $p$ is nonforking over $B$. Let $q\in\mathbf{S}(\mathbb{I}_{\leq\alpha(\ast)})$
a nonforking extension of $p\upharpoonright\mathbb{I}_{<\alpha(\ast)}$.

By transitivity of non-forking, $q$ is nonforking over $B$. Assume
towards contradiction that $q\neq p$. Then by \ref{thm:nonforking_contradicting_types}
there exists an $E\in FE(B)$ such that $q(x)\cup p(y)\vdash\neg E(x,y)$,
and particularly $q(x)\vdash\neg E(x,a)$. 

The formula $E(x,a)$ is almost over $B$, therefore by the choice
of $\alpha_{1}(\ast)$, there exists a $\varphi(x,\bar{b})$ logically
equivalent to $E(x,a)$ in $T$, with $\bar{b}\subseteq\mathbb{I}_{<\alpha(\ast)}$.

Now, since $E(a,a)$, it also holds that $\models\varphi(a,\bar{b})$,
and $\bar{b}\subseteq\mathbb{I}_{<\alpha(\ast)}$ implies $\varphi(x,\bar{b})\in\mathrm{tp}(a,\mathbb{I}_{<\alpha(\ast)})=q\upharpoonright\mathbb{I}_{<\alpha(\ast)}$
, a contradiction.

In particular, $\mathrm{tp}(a,\mathbb{I}_{\leq\alpha(\ast)})$ is
the only nonforking extension of $\mathrm{tp}(a,\mathbb{I}_{<\alpha(\ast)})$
in $\mathbf{S}\left(\mathbb{I}_{\leq\alpha(\ast)}\backslash\left\{ b\right\} \right)$
. By the choice of $\mathbb{I}_{\alpha(\ast)}$ it follows for all
$b\in\mathbb{I}_{\alpha(\ast)}$ that $\mathrm{tp}(b,\mathbb{I}_{\leq\alpha(\ast)}\backslash\left\{ b\right\} )$
is the only nonforking extension of $\mathrm{tp}(b,\mathbb{I}_{<\alpha(\ast)})$
in $\mathbf{S}\left(\mathbb{I}_{\leq\alpha(\ast)}\backslash\left\{ b\right\} \right)$.

\relax From claim \ref{cla:str_indep_symmetry} it follows that $\mathrm{tp}(b,\mathbb{I}_{\leq\alpha(\ast)}\backslash\left\{ b\right\} \cup\left\{ a\right\} )$
is the only nonforking extension of $\mathrm{tp}(b,\mathbb{I}_{<\alpha(\ast)})$
in $\mathbf{S}(\mathbb{I}_{\leq\alpha(\ast)}\backslash\left\{ b\right\} \cup\left\{ a\right\} )$.

So, $\circledast$ holds for $\mathbb{I}_{\alpha(\ast)}\cup\left\{ a\right\} $
 ( with respect to $\mathbb{I}_{<\alpha(\ast)}$) contradicting the
maximality of $\mathbb{I}_{\alpha(\ast)}$.
\end{proof}
\begin{claim}
\label{cla:Forking-is-elementary}Forking is preserved under elementary
maps \cite{Sh:c}*{III.1.5}
\end{claim}

\theorem{\label{thm:definability-for-types}definability for types cf. 
\cite{Sh:c}*{II;2.2}}

\subsection{Representing stable theories.}

\theorem{If $M\models T$, then $M$ is ${\rm Ex}_{\left|T\right|^{+},\left|T\right|}^{1}(\mathfrak{k}^{{\rm eq}})$-representable.}

\begin{proof}
By \ref{lem:stable_exist_decomp} we get a strongly independent decomposition
of $M$: $\left\langle \mathbb{I}_{\alpha}:\alpha<\left|T\right|^{+}\right\rangle $.
By Claim \ref{cla:sub_decompo} we assume w.l.o.g $\left|\mathbb{I}_{1}\right|=\left|\mathbb{I}_{0}\right|=1$.

Define the structure $I^{+}\in{\rm Ex}_{\left|T\right|^{+},\left|T\right|}^{1}(\mathfrak{k}^{\mathrm{eq}})$
as follows:
\begin{enumerate}
\item $\left|I^{+}\right|=\left|M\right|$.
\item for all $\alpha<\left|T\right|^{+}$, $P_{\alpha}^{I^{+}}=\mathbb{I}_{\alpha}$
. 
\item for all $\varphi(x,\bar{y})\in\mathcal{L}_{M}$ define $n$ one-place
partial functions ( let $n=\lg\bar{z}$ ) $\left\{ F_{\varphi(x,\bar{y}),j}^{I^{+}}(x):j<n\right\} $
as follows: 

\begin{enumerate}
\item $\mathrm{Dom}F_{\varphi(x,\overline{y}),j}^{I^{+}}=\left|M\right|\backslash\left(\mathbb{I}_{0}\cup\mathbb{I}_{1}\right)$.
\item By Theorem \ref{thm:definability-for-types} we get for every $\varphi(x,\bar{y)}\in\mathcal{L}_{M}$
another formula $\psi_{\varphi}(\bar{y},\bar{z})\in\mathcal{L}_{M}$,
such that for all $2\leq\alpha<\mu$, $a\in\mathbb{I}_{\alpha}$ there
exists $\bar{c}_{a}\in^{\lg\bar{z}}\mathbb{I}_{<\alpha}$ such that
for all $\bar{b}\in\mathbb{I}_{\alpha}$, $\models\varphi[a,\bar{b}]\Leftrightarrow\models\psi_{\varphi}[\bar{b},\bar{c}_{a}]$
holds.
\item For all $2\leq\alpha<\mu$ and $a\in\mathbb{I}_{\alpha}$, let $F_{\varphi(x,\bar{y}),j}^{I^{+}}(a):=\left(\bar{c}_{a}\right)_{j}$
\end{enumerate}
\item Add $\left|T\right|$ partial functions $\left\langle \left(F_{i}^{*}\right)^{I^{+}}:i<\left|T\right|\right\rangle $
as follows:

\begin{enumerate}
\item $\mathrm{Dom}F_{i}^{\ast}=\left|M\right|\backslash\mathbb{I}_{<2}$
\item Fix $\alpha>1$, then there exists $\left|B\right|\leq\left|T\right|$
such that for every $\varphi(\overline{x},\overline{c})$ over $\mathbb{I}_{<\alpha}$
which is almost over $B$ there exists a $\theta(\overline{x},\overline{d})$
over $B$ such that $\models\forall\overline{x}\left(\theta(\overline{x},\overline{d})\leftrightarrow\varphi(\overline{x},\overline{c})\right)$:

\begin{enumerate}
\item Let $\left|B_{0}\right|<\kappa(T)\leq\left|T\right|^{+},\; B_{0}\subseteq\mathbb{I}_{<\alpha}$
such that $\mathrm{tp}(a,\mathbb{I}_{<\alpha})$ does not fork over
$B_{0}$.
\item Assume $B_{n}$ is defined and let\[
B_{n+1}:=B_{n}\cup\left\{ \overline{c}:\varphi(\overline{x},\overline{c})\in S^{\prime}\right\} \]
where $S^{\prime}$ is a complete set of representatives of $S$,
relative to logical equivalence in $T$\[
S:=\left\{ \varphi(\overline{x},\overline{c})\in\mathcal{L}_{T}:\overline{c}\subseteq\mathbb{I}_{<\alpha},\;\varphi\text{ is almost over }B_{n}\right\} \]
by \ref{cla:count_almost_over} we can assume w.l.o.g $\left|S^{\prime}\right|\leq\left|T\right|+\left|B_{n}\right|=\left|T\right|$.
\item Then the set $B=\bigcup_{n<\omega}B_{n}$ is a s required.
\end{enumerate}
\item Let $\left\langle b_{i}:i<\left|T\right|\right\rangle $ enumerate
$B$ (possibly with repetitions). We define $F_{i}(a)=b_{i}$.
\end{enumerate}
\item Let $f:M\to I^{+}$ be defined as $f(a)=a$ for all $a\in\left|M\right|$.
\end{enumerate}
Let $h$ be a partial automorphism of $I^{+}$ whose domain and range
are closed under partial functions of $I^{+}$.

\subsubsection*{We show that $\mathrm{tp}(h(\bar{a}),\emptyset,M)=\mathrm{tp}(\bar{a},\emptyset,M)$
holds for all $\bar{a}\subseteq\mathrm{Dom}(h)$:}

\begin{itemize}
\item It is sufficient to show for all $\alpha<\left|T\right|^{+},\, n<\omega$,
$\bar{a}\in\mathbb{I}_{\alpha}\cap\mathrm{Dom}(h)$ without repetitions
($n:=\lg\bar{a}$) the following holds: \[
h\left(\mathrm{tp}(\bar{a},\mathbb{I}_{<\alpha}\cap\mathrm{Dom}(h))\right)=\mathrm{tp}(h(\bar{a}),\mathbb{I}_{<\alpha}\cap\mathrm{Rang}(h))\qquad\boxtimes_{\alpha,n}\]
\\
we prove this by induction on the lexicographic well-order $\left|T\right|^{+}\times\omega$
\item For $\boxtimes_{\alpha,n}$ holds for $\alpha<2$ since $\mathbb{I}_{\alpha}$
is a singleton.
\item Let $\alpha\geq2$, and assume $\boxtimes_{\beta,n}$ for all $n<\omega$
and $\beta<\alpha$.
\item $\boxtimes_{\alpha,1}$ holds, since let $a\in\mathbb{I}_{\alpha}$,
$\varphi(x,\bar{c})$ a formula over $\mathrm{Dom}(h)\cap\mathbb{I}_{<\alpha}$
such that $\varphi[a,\bar{c}]$ holds. Then by the definitions of
the $F$'s above $\psi_{\varphi}[\bar{c},F_{\varphi,0}(a)\ldots F_{\varphi,\lg\bar{y}-1}(a)]$
holds. Since the latter is a formula over $\mathrm{Dom}(h)\cap\mathbb{I}_{<\alpha}$
and by the induction hypothesis it follows that $\psi_{\varphi}[h(\bar{c}),h(F_{\varphi,0}(a))\ldots h(F_{\varphi,\lg\bar{y}-1}(a))]$
holds. Also by the induction hypothesis, $h$ commutes with the functions
of $I^{+}$ over the domain $\mathbb{I}_{<\alpha}\cap\mathrm{Dom}(h)$.
Hence, $\psi_{\varphi}[h(\bar{c}),F_{\varphi,0}(h(a))\ldots F_{\varphi,\lg\bar{y}-1}(h(a))]$
holds. The definition of $F_{\varphi,j}(x)$ implies that $M\models\varphi[h(a),h(\overline{c})]$.
\item For $n>1$ We continue by induction, but first we prove the following:

\begin{claim}
Let $A\subseteq I^{+}$ be closed under functions of $I^{+}$. Then
$A\cap\mathbb{I}_{\alpha}$ is strongly independent over $A\cap\mathbb{I}_{<\alpha}$.
\end{claim}
\begin{proof1}
Let $A_{\alpha}=\mathbb{I}_{\alpha}\cap A$, $a\in A_{\alpha}$, $B:=\left\{ F_{i}^{\ast}(a):i<\left|T\right|\right\} $.
Then,
\begin{enumerate}
\item $B\subseteq A_{<\alpha}$.
\item By the choice of the $F_{i}^{\ast}$'s it holds that $\mathrm{tp}(a,\mathbb{I}_{<\alpha})$
is nonforking over $B$ 
\item By 2 and by transitivity of nonforking, $\mathrm{tp}(a,\mathbb{I}_{\leq\alpha}\backslash\left\{ a\right\} )$
is nonforking over $B$. $\mathrm{tp}(a,\mathbb{I}_{\leq\alpha}\backslash\left\{ a\right\} )$
is a nonforking extension of $\mathrm{tp}(a,\mathbb{I}_{<\alpha})$.
\item For any formula over $\mathbb{I}_{<\alpha}$ which is almost over
$B$ there exists an equivalent formula (in $T$) over $B$ (by the
choice of the $F_{i}^{\ast}$)
\end{enumerate}
The first two properties imply that $\mathrm{tp}(a,A_{\leq\alpha}\backslash\left\{ a\right\} )\subseteq\mathrm{tp}(a,\mathbb{I}_{\leq\alpha}\backslash\left\{ a\right\} )$
is nonforking over $A_{<\alpha}$. 

We turn to proving the uniqueness. Let $q_{0}\in\mathbf{S}(A_{\leq\alpha}\backslash\left\{ a\right\} )$
be a nonforking extension of $\mathrm{tp}(a,A_{<\alpha})$.
\begin{itemize}
\item $q_{0}$ has a nonforking extension $q\in\mathbf{S}(\mathbb{I}_{\leq\alpha}\backslash\left\{ a\right\} )$.
\item $q$ is nonforking over $A_{\leq\alpha}\backslash\left\{ a\right\} $
and by transitivity nonforking over $A_{<\alpha}$ and therefore nonforking
over $A_{<\alpha}\subseteq\mathbb{I_{<\alpha}}$.
\item $q\upharpoonright\mathbb{I}_{<\alpha}=\mathrm{tp}(a,\mathbb{I}_{<\alpha})$
- since otherwise, a formula $\varphi(x)$ over $\mathbb{I}_{<\alpha}$
exists such that $q(x)\vdash\varphi(x),\;{\rm tp}(a,\mathbb{I}_{<\alpha})\vdash\neg\varphi(x)$\.{ }
By 4 above ( as $B$ was chosen ) and claim \ref{cla:nonforking_contr_contr}
$\varphi(x)$ is equivalent to a formula over $B$. Hence, $q\upharpoonright B\neq{\rm tp}(a,B)$
contradicting the choice of $q$.
\item So, $q$ is a nonforking extension of $q\upharpoonright\mathbb{I}_{<\alpha}$,
unique by the strong independence of $\mathbb{I}_{\alpha}$ over $\mathbb{I}_{<\alpha}$,
and therefore equal to ${\rm tp}(a,\mathbb{I}_{\leq\alpha}\backslash\left\{ a\right\} )$.
\item The above arguments imply the required conclusion - $q_{0}=q\upharpoonright(A_{\leq\alpha}\backslash\left\{ a\right\} )=\mathrm{tp}(a,A_{\leq\alpha}\backslash\left\{ a\right\} )$
\end{itemize}
\end{proof1}
\item We continue the main proof, letting $D_{\gamma}:=\mathrm{Dom}(h)\cap\mathbb{I}_{\gamma},\; R_{\gamma}:=\mathrm{Rang}(h)\cap\mathbb{I}_{\gamma}$
( for all $\gamma<\left|T\right|^{+}$, $h"(D_{\gamma})=R_{\gamma}$).

Let $\bar{a}\in^{n}\left(D_{\alpha}\right)$ and $b\in D_{\alpha}\backslash\bar{a}$. 

\begin{itemize}
\item $h\upharpoonright\left(D_{<\alpha}\cup\bar{a}\right)$ is elementary
by the induction hypothesis.
\item $\mathrm{tp}(b,D_{\leq\alpha}\backslash\left\{ b\right\} )$ does
not fork over $D_{<\alpha}$ (by the last claim, and since $\mathrm{Dom}(h)$
us closed under functions), therefore $\mathrm{tp}(b,D_{<\alpha}\cup\bar{a})$
also does not fork over $D_{<\alpha}$.
\end{itemize}
The above, with claim \ref{cla:Forking-is-elementary} imply that
$q:=h(\mathrm{tp}(b,D_{<\alpha}\cup\bar{a}))$ does not fork over
$h(\mathrm{Dom}(h)\cap\mathbb{I}_{<\alpha})=\mathrm{Rang}(h)\cap\mathbb{I}_{<\alpha}$.

\begin{itemize}
\item $\boxtimes_{\beta,1}$ holds for all $\beta<\alpha$, and in particular
$q\in\mathbf{S}(R_{<\alpha}\cup\bar{a})$ is a nonforking extension
of $\mathrm{tp}(h(b),R_{<\alpha})$. Also, $q$ has a nonforking extension
$q^{\prime}\in\mathbf{S}(R_{\leq\alpha}\backslash h(b))$ which does
not fork over $R_{<\alpha}$ by transitivity.
\item On the other hand, since $\mathrm{Rang}(h)$ is closed under functions
and by the last claim, it follows that $R_{\alpha}$ is strongly independent
over $R_{<\alpha}$. hence, $q^{\prime}=\mathrm{tp}(h(b),R_{\leq\alpha}\backslash\left\{ h(b)\right\} )$.
After reduction to $R_{<\alpha}\cup h(\bar{a})$ we get \[
\mathrm{tp}(h(a),R_{<\alpha}\cup h(\bar{a}))=h(\mathrm{tp}(b,D_{<\alpha}\cup\bar{a})\]
 implying the inductive step from$\left(\alpha,n\right)$ to $\left(\alpha,n+1\right)$\[
\mathrm{tp}(h(b\frown\bar{a}),R_{<\alpha})=h(\mathrm{tp}(b\frown\bar{a},D_{<\alpha}))\]

\end{itemize}
\end{itemize}
\end{proof}

\subsection{Representation for $\omega$-stable theories}

\convention{For the remainder of the section $T$ is $\omega$-stable}

\claim{Let $p\in\mathbf{S}(A)$. Then, there exists a finite $B\subseteq A$ 
 such that $p$ is a nonforking extension of $p\upharpoonright B$.}
(See: \cite{Sh:c})

\claim{\label{cla:strong_indep_kappa_a0_stable}For every $p\in\mathbf{S}(A)$
there exists a finite $B\subseteq A$ such that $p$ is the unique nonforking extension of $p\upharpoonright B$ in $\mathbf{S}(A)$.

\claim{\label{cla:special_decomp}Let $M\models T$. $M$ has a strongly independent decomposition $\left\langle \mathbb{I}_{n}:n<\omega\right\rangle $, so that}

\begin{enumerate}
\item $\mathbb{I}_{0}$ is an indiscernible set over $\emptyset$ ( possibly finite ), and
\item For every $a\in\mathbb{I}_{n},\; n<\omega$ there exists a finite $B_{a}\subseteq\mathbb{I}_{<n}$
so that $\mathrm{tp}(a,\mathbb{I}_{\leq n}\backslash\left\{ a\right\} )$
is the unique nonforking extension of $\mathrm{tp}(a,B_{a})$ in $\mathbf{S}(\mathbb{I}_{\leq n}\backslash\left\{ a\right\} )$.
\end{enumerate}
\begin{proof}
The first condition is fulfilled by a singleton, so it is 
possible to find a $\mathbb{I}_{0}\subseteq\left|M\right|$ as above.
For $n>0$, Construct a sequence $\left\langle \mathbb{I}_{n}:n<\omega\right\rangle $ 
such that $\mathbb{I}_{n}\subseteq\left|M\right|$ is 
maximal with respect to the second condition ( possibly empty ) for every $n<\omega$.
Assume towards contradiction that there exists 
$a\in M\backslash\mathbb{I}_{<\omega}$.
By \ref{cla:strong_indep_kappa_a0_stable} it follows that there exists a finite 
$B_{a}\subseteq\mathbb{I}_{<\omega}$ such that $\mathrm{tp}(a,\mathbb{I}_{<\omega})$ 
is the unique nonforking extension of $\mathrm{tp}(a,B_{a})$ in $\mathbf{S}(\mathbb{I}_{<\omega})$.
Clearly, this implies that $\mathbb{I}_n\neq\emptyset$ for all $n<\omega$.
Therefore, there exists $0<n_{\ast}<\omega$ such that $B_{a}\subseteq\mathbb{I}_{<n_{\ast}}$.
In particular it follows that  $\mathrm{tp}(a,\mathbb{I}_{\leq n_{\ast}})$ is 
the unique nonforking extension of $\mathrm{tp}(a,B_{a})$ in $\mathbf{S}(\mathbb{I}_{\leq n_{\ast}})$
( otherwise, by transitivity of nonforking we would have two nonforking extensions in $\mathbf{S}(\mathbb{I}_{<\omega})$). 

The construction above implies that there exists a finite $B_{b}\subseteq\mathbb{I}_{<n_{\ast}}$ such that $\mathrm{tp}(b,\mathbb{I}_{\leq n_{\ast}}\backslash\left\{ b\right\} )$
is the unique nonforking extension of $\mathrm{tp}(b,B_{b})$ in $\mathbf{S}(\mathbb{I}_{\leq n_{\ast}}\backslash\left\{ b\right\} )$.
Claim \ref{cla:str_indep_symmetry} implies
that for every $b\in\mathbb{I}_{n_{\ast}}$,
$\mathrm{tp}(b,\mathbb{I}_{\leq n_{\ast}}\backslash\left\{ b\right\} \cup\left\{ a\right\} )$
is the unique nonforking extension of $\mathrm{tp}(b,B_{b})$ in $\mathbf{S}(\mathbb{I}_{\leq n_{\ast}}\backslash\left\{ b\right\} \cup\left\{ a\right\} )$,
Thus, $\mathbb{I}_{n_{\ast}}\cup\left\{ a\right\} $ fulfills
the second condition, contradicting the maximality of $\mathbb{I}_{n_{\ast}}$.

\end{proof}

\theorem{Let $M\models T$, then $M$ is ${\rm Ex}_{\omega,\omega}^{2}(\mathfrak{k}^{{\rm eq}})$-representable.}

\begin{proof}
Let $\left\langle \mathbb{I}_{n}:n<\omega\right\rangle $ as in \ref{cla:special_decomp}, $I=\left|\mathbb{I}_{0}\right|$.
Since $T$ is $\omega$-stable, $\mathbf{S}^{m}(\emptyset)$
is countable for all $m<\omega$. 
For convenience we replace the enumeration of the functions of $\mathcal{M}(I)$ to 
$\left\{ F_{p}:p\in\mathbf{S}^{<\omega}(\emptyset)\right\} $,
and for every $m+1$-type $F_{p}$ is an $m$-ary function.
Define by induction an increasing series of functions $f_{i}:\mathbb{I}_{\leq i}\to\mathcal{M}(I)$ as follows:
Let $f_{0}$ be a bijective map from $\mathbb{I}_{0}$ onto $I$.
Let $f_{n+1}$ be defined from $f_{n}$ as follows:
\begin{itemize}
\item $f_{n+1}\upharpoonright\mathbb{I}_{\leq n}=f_{n}$
\item For all $a\in\mathbb{I}_{n+1}$, let $\bar{c}_{a}\in^{\ell}\left(\mathbb{I}_{\leq n}\right)$
enumerate $B_{a}$ from Claim \ref{cla:strong_indep_kappa_a0_stable},
$p=\mathrm{tp}(a\!^{\frown}\bar{c}_{a},\emptyset,M)\in\mathbf{S}^{\ell+1}(\emptyset)$.
Define $f_{n+1}(a)=F_{p}(f_{n}(\bar{c}_{a}))$. Now let $f=\bigcup_{n<\omega}f_{n}$. 
Now we will show that $f$ is a $\mathrm{Ex}_{\omega,\omega}^{2}(\mathfrak{k}^{{\rm eq}})$-representation.
Let $h$ be a partial automorphism of $\mathcal{M}(I)$ with domain and range closed under subterms.
Let $\overline{a},\overline{b}\in M$ such that $h(f(\overline{a}))=f(\overline{b})$, and $n$ so that $\overline{a},\overline{b}\in\mathbb{I}_{\leq n}$.
Assume w.l.o.g $m<\omega,\; i<\lg\overline{a}-1$: $a_{\lg\overline{a}-1},b_{\lg\overline{b}-1}\in\mathbb{I}_{\leq m}\to a_{i},b_{i}\in\mathbb{I}_{\leq m}$.
We prove ${\rm tp}_{{\rm qf}}(\overline{a},\emptyset)={\rm tp}_{{\rm qf}}(\overline{b},\emptyset)$
by induction on $\left\langle n,\left|\overline{a}\cap\mathbb{I}_{n}\right|\right\rangle \in\omega\times\omega$.
\end{itemize}
\begin{description}
\item [{case~$n=0$:}] the claim holds since $\mathbb{I}_{0}$ is an indiscernible set.
\item [{case~$n=m+1$:}]~
\begin{description}
\item [{case~$\left|\overline{a}\cap\mathbb{I}_{n}\right|=0$:}] $\overline{a}\subseteq\mathbb{I}_{\leq m}$, hence,
the claim holds by the induction hypothesis.
\item [{case~$\left|\overline{a}\cap\mathbb{I}_{n}\right|>0$:}] Let
$k=\lg\overline{a}-1$. By the definition, $f(a_{k})=F_{p}(\overline{c}_{a_{k}})$
where $\overline{c}_{a_{k}}\subseteq\mathbb{I}_{<n}$.
$h$ commutes with $F_{p}$, implying $f(b_{k})=h(f(a_{k}))=h(F_{p}(f(\overline{c}_{a_{k}})))=F_{p}(h(f(\overline{c}_{a_{k}})))$. 
Therefore, $h(f(\overline{c}_{a_{k}}))=f(\overline{c}_{b_{k}})$. 
Now, since $\left|\overline{c}_{a_{k}}\!^{\frown}\overline{a}\upharpoonright k\cap\mathbb{I}_{n}\right|=\left|\overline{a}\cap\mathbb{I}_{n}\right|-1$, and by the induction hypothesis, 
the map $F:\overline{c}_{a_{k}}\!^{\frown}\overline{a}\upharpoonright k\mapsto\overline{c}_{b_{k}}\!^{\frown}\overline{b}\upharpoonright k$ is elementary.
Consider the type $q=F\left({\rm tp}\left(a_{k},\overline{a}\upharpoonright k\cup\overline{c}_{a_{k}}\right)\right)$.
Note ${\rm tp}(a_{k}\!^{\frown}\overline{c}_{a_{k}})=p={\rm tp}(b_{k}\!^{\frown}\overline{c}_{b_{k}})$,
so $F({\rm tp}(a_{k},\overline{c}_{a_{k}}))={\rm tp}(b_{k},\overline{c}_{b_{k}})$.
Then, $q$ is a nonforking extension of ${\rm tp}(b_{k},\overline{c}_{b_{k}})$.
Moreover, $F$ being elementary and ${\rm tp}(a_{k},\overline{a}\upharpoonright k\cup\overline{c}_{a_{k}})$
is a nonforking extension of ${\rm tp}(a_{k},\overline{c}_{a_{k}})$ imply
that $q$ is a nonforking extension of ${\rm tp}(b_{k},\overline{c}_{b_{k}})$.
Now let $q\subseteq q^{\prime},{\rm tp}(b_{k},\overline{b}\upharpoonright k\cup\overline{c}_{b_{k}})\subseteq q^{\prime\prime},\quad q^{\prime},q^{\prime\prime}\in\mathbf{S}(\mathbb{I}_{\leq n}\backslash\left\{ b_{k}\right\} )$ be nonforking extensions.
By monotonicity of nonforking extensions, $q^{\prime},q^{\prime\prime}$
are nonforking extensions of ${\rm tp}(b_{k},\overline{c}_{b_{k}})$.
The definition of $\overline{c}_{b_{k}}$ implies $q^{\prime}=q^{\prime\prime}$.
Thus, $q={\rm tp}(b_{k},\overline{b}\upharpoonright k\cup\overline{c}_{b_{k}})$,
therefore ${\rm tp}(\overline{a}\!^{\frown}\overline{c}_{a_{k}})={\rm tp}(\overline{b},\overline{c}_{b_{k}})$.
${\rm tp}(\overline{a})={\rm tp}(\overline{b})$ follows.
\end{description}
\end{description}
\end{proof}

\section{Appendix - combinatorial claims.}

\theorem{(Fodor) Let $\lambda$ a regular cardinal, and $f:\lambda\to\lambda$ such that $f(\alpha)<\alpha$ for all $0<\alpha<\lambda$. (such $f$ is called regressive) Then there exists an ordinal $\beta<\lambda$ such that the set $\left\{ \alpha<\lambda:f(\alpha)=\beta\right\} $ is
stationary in $\lambda$.}

\corollary{Let $f:\lambda\to\mu$, $\lambda>\mu$ ($\lambda$ regular). There exists an $\alpha<\mu$ such that $f^{-1}(\left\{ \alpha\right\} ) \subseteq\lambda$ is stationary}

\theorem{\label{thm:delta_sys_lemma}($\Delta$-system Lemma) Let $\lambda$
 regular, $\left|W\right|=\lambda$ a set, $\left|S_{t}\right|<\mu\;(t\in W)$
such that $\chi^{<\mu}<\lambda$ for all $\chi<\lambda$. then:} 

\begin{enumerate}
\item There exist $W^{\prime}\subseteq W,\;\left|W^{\prime}\right|=\lambda$
and $S$ such that $s\neq t$ implies $S_{t}\cap S_{s}=S$ for all $s,t\in W^{\prime}$.
\item Moreover, if $\left\langle z_{t}^{\alpha}:\alpha<\alpha(t)\right\rangle $ lists $S_{t}$, also:

\begin{enumerate}
\item \label{delta:a} There exists $\alpha_{0}$ such that $\alpha(t)=\alpha_{0}$ for all $t\in W^{\prime}$.
\item \label{delta:b} There exists $U\subseteq\alpha_{0}$ such that for all $s,t\in W^{\prime}$ implies $S_{t}\upharpoonright U=S_{s}\upharpoonright U$,
$U=\left\{ \alpha<\alpha_{0}:z_{t}^{\alpha}=z_{s}^{\alpha}\right\} $.
\item \label{delta:c} There exists an equivalence $E$ on $\alpha_{0}$ such that $z_{t}^{\alpha}=z_{t}^{\beta}\leftrightarrow\left(\alpha,\beta\right)\in E$,
for all $t\in W^{\prime}$.
\end{enumerate}

\begin{proof}
Proofs for the first part can be found in \cite{J}.

The map $t\to\alpha(t)$ is regressive ($\alpha(t)<\mu<\lambda$),
so by Fodor's theorem there exists $W_{0}\subseteq W$ such that \ref{delta:a} holds. 
By the first part there exists $S\subseteq\left\{ z_{t}^{\alpha}:\alpha<\alpha_{0},t\in W_{0}\right\} ,\; W_{1}\subseteq W_{0}$
such that $S=\overline{z}_{t}\cap\overline{z}_{s}$ for all $t\neq s$. 
Define the map map $W_{1}\ni t\to U_{t}$ where $U_{t}=\left\{ \alpha<\alpha_{0}:z_{t}^{\alpha}\in S\right\} $.
The range has power at most $2^{\left|\alpha_{0}\right|}\leq2^{<\mu}<\lambda$
implying that the map is regressive, and the existence of 
$W_{2}\subseteq W_{1},\; U$ such that $t\in W_{2}\to U_{t}=U$. 
The range of the map $t\to S_{t}\upharpoonright U$
is $\ ^{U}S$ and it has power $\leq\left|\alpha_{0}\right|^{\left|\alpha_{0}\right|}<\lambda$,
By another use of Fodor's theorem there exists $W_{3}\subseteq W_{2}$ such that
$(b)$ holds. The range of the map $t\to E_{t}$ where $E_{t}=\left\{ \left(\alpha,\beta\right):z_{t}^{\alpha}=z_{t}^{\beta},\;\alpha,\beta<\alpha_{0}\right\} $
has power at most $\left|\alpha_{0}\right|^{\left|\alpha_{0}\right|}$
And by another application of Fodor's theorem there are $E$ and $W^{\prime}\subseteq W_{3}$ as required.
\end{proof}
\end{enumerate}

\addcontentsline{toc}{section}{\refname}

\begin{bibdiv}
\begin{biblist}

\bib{J}{book}{
      author={Jech, Thomas},
       title={{Set theory}},
   publisher={Academic Press, New York},
        date={1978},
}

\bib{Sh:c}{book}{
      author={Shelah, Saharon},
       title={Classification theory and the number of nonisomorphic models},
      series={Studies in Logic and the Foundations of Mathematics},
   publisher={North-Holland Publishing Co., Amsterdam, xxxiv+705 pp},
        date={1990},
      volume={92},
}

\end{biblist}
\end{bibdiv}


\end{document}